\documentclass[final,reqno]{siamltex}
\usepackage{latexsym,amsmath,amssymb,amsfonts,mathrsfs}
\usepackage{epsf,graphicx,epsfig,epstopdf,color,cite,cases}
\usepackage{subfigure,graphics,multirow,marginnote,graphicx,enumerate,bm}
\newcommand{\R}{{\mathbb R}}

\newcommand{\C}{{\mathbb C}}

\newcommand{\E}{{\mathbb E}}
\newcommand{\ds}{\displaystyle}
\newcommand{\no}{\nonumber}
\newcommand{\be}{\begin{eqnarray}}
\newcommand{\ben}{\begin{eqnarray*}}
\newcommand{\en}{\end{eqnarray}}
\newcommand{\enn}{\end{eqnarray*}}

\newcommand{\pa}{\partial}

\newcommand{\ov}{\overline}

\newcommand{\dive}{{\rm div\,}}
\newcommand{\I}{{\rm Im}}
\newcommand{\Rt}{{\rm Re}}
\newcommand{\g}{\gamma}
\newcommand{\G}{\Gamma}

\newcommand{\vep}{\varepsilon}
\newcommand{\Om}{\Omega}
\newcommand{\om}{\omega}
\newcommand{\sig}{\sigma}

\newcommand{\la}{\lambda}

\newcommand{\wi}{\widehat}
\newcommand{\wid}{\widetilde}

\newcommand{\ra}{\rightarrow}
\newcommand{\se}{\setminus}

\newcommand{\ify}{\infty}

\newcommand{\on}{\text{on}}
\newcommand{\gin}{\text{in}}
\newcommand{\n}{\bm{n}}

\newcommand{\0}{\bm{0}}
\newcommand{\ch}{\check}
\newcommand{\na}{\nabla}
\newtheorem{remark}[theorem]{Remark}
\begin{document}
\renewcommand{\theequation}{\arabic{section}.\arabic{equation}}
\title{\bf
Convergence of the perfectly matched layer method for transient acoustic-elastic interaction
above an unbounded rough surface
}
\author{Changkun Wei\thanks{Academy of Mathematics and Systems Science,
Chinese Academy of Sciences, Beijing 100190, China and School of Mathematical Sciences,
University of Chinese Academy of Sciences, Beijing 100049, China ({\tt weichangkun@amss.ac.cn})}
\and
Jiaqing Yang\thanks{School of Mathematics and Statistics, Xi'an Jiaotong University,
Xi'an, Shaanxi, 710049, China ({\tt jiaq.yang@mail.xjtu.edu.cn; jiaqingyang@amss.ac.cn})}
\and
Bo Zhang\thanks{LSEC and Academy of Mathematics and Systems Sciences, Chinese Academy of Sciences,
Beijing 100190, China and School of Mathematical Sciences, University of Chinese Academy of Sciences,
Beijing 100049, China ({\tt b.zhang@amt.ac.cn})}
}
\date{}

\maketitle

\begin{abstract}
This paper is concerned with the time-dependent acoustic-elastic interaction problem associated with a
bounded elastic body immersed in a homogeneous air or fluid above an unbounded rough surface.
The well-posedness and stability of the problem are first established by using the Laplace transform
and the energy method. A perfectly matched layer (PML) is then introduced to truncate the interaction
problem above a finite layer containing the elastic body, leading to a PML problem in a finite strip domain.
We further establish the existence, uniqueness and stability estimate of solutions to the PML problem.
Finally, we prove the exponential convergence of the PML problem in terms of the thickness and parameter
of the PML layer, based on establishing an error estimate between the DtN operators of the original problem
and the PML problem.
\end{abstract}



\pagestyle{myheadings}
\thispagestyle{plain}
\markboth{C. Wei, J. Yang and B. Zhang}{PML for transient acoustic-elastic interaction
above a rough surface}

\section{Introduction}
\setcounter{equation}{0}

Consider the problem of scattering of acoustic waves by an elastic body immersed in a compressible,
inviscid fluid (air or water) in a half-space with an unbounded rough boundary.
This problem is also refereed to as a fluid-solid interaction problem which can be mathematically
formulated as an initial-boundary transmission problem and has been widely studied (see, e.g. \cite{HKS1989,Hsiao1994,MORAND1995,Luke1995,HKF2000,Hsiao2015,GLZ2017,Bao2018} and the references
quoted there).
This problem can also be categorized into the class of unbounded rough surface scattering problems,
which is the subject of intensive studies in the engineering and mathematics communities.
For the rough surface scattering problems, the usual Sommerfeld radiation condition and
Silver-M\"{u}ller radiation condition is not valid anymore due to the unbounded structure.
We refer to \cite{CB1998,CRZ1999,CM2005,CHP2006} for the mathematical analysis of
the time-harmonic case using both the integral equation method and the variational method.

In most of real-world problems, the model setting not only depends on the space, but also depends on
the time. Recently, this class of problems has attracted much attention due to their capability of
capturing wide-band signals and modeling more general material and nonlinearity
(see, e.g. \cite{Li2012,Chen2014,Wang2012,Wang2014} and the references quoted there).
In particular, the analysis of time-dependent scattering problems can be found in \cite{chen2009,Wang2014}
for the acoustic case, in \cite{chen2008,LLA2015,GL2016,GL2017} for the electromagnetic case
including the cases with bounded obstacles, diffraction gratings and unbounded surfaces,
and in \cite{Bao2018,GLZ2017,Hsiao2015,wy2019} for the time-dependent fluid-solid interaction problems
including the cases with bounded elastic bodies \cite{Bao2018,Hsiao2015},
locally rough surfaces \cite{wy2019} and unbounded layered structures \cite{GLZ2017}.

The perfectly matched layer (PML) method is a fast and effective method for solving unbounded 
scattering problems which was originally proposed by B\'erenger in 1994 for electromagnetic 
scattering problems \cite{Berenger1994}.
A large amount of work have been done since then to construct various PML absorption
layers \cite{TURKEL1998,Lassas1998,Teixeira2001,CW2003,CM2009,chen2009}.
The key idea of the PML method is to surround the computational domain with a specially designed
medium containing a finite thickness layer in which the scattered waves decay rapidly regardless
of the frequencies and incident angles, thereby greatly reducing the computational complexity of
the scattering problems. This makes the PML method a popular approach to solve a variety of
wave scattering problems \cite{TURKEL1998,Collino1998,BW2005,Bramble2007}.

The convergence of the PML method has always been a topic of interest to mathematicians.
There are a lot of works on the convergence of the time-harmonic PML method, most of which
focus on the exponential convergence of the PML method in terms of the thickness of the PML
layer for the case of bounded scatterers (see, e.g., \cite{Lassas1998,HSZ2003,Chen2005,BW2005,CZ2010}).
In 2009, Chandler-Wilde and Monk extend the PML method to time-harmonic scattering problems
by unbounded rough surfaces in \cite{CM2009}, where only the linear convergence of the PML method
was established in terms of the PML layer thickness.

Compared with the time-harmonic case, only several results are available for the rigorous convergence
analysis of the PML method for time-domain wave scattering problems.
For the case of time-domain acoustic scattering by a bounded scatterer, the exponential convergence
in terms of the thickness and parameter of the PML layer was proved in \cite{chen2009} for
a circular PML method and in \cite{Chen2012} for an uniaxial PML method.
The method used in \cite{chen2009,Chen2012} is based on the Laplace transform and complex
coordinate stretching technique.
For the case of time-domain electromagnetic scattering by bounded scatterers,
the exponential convergence of a spherical PML method was recently shown in \cite{wyz2019}
in terms of the thickness and parameter of the PML layer, based on a real coordinate stretching
technique associated with $[\Rt(s)]^{-1}$ in the Laplace domain, where $s\in\C_+$ is the
Laplace transform variable. Recently, a time-domain PML method was studied in \cite{Bao2018}
for the transient acoustic-elastic interaction problem, where a bounded elastic body is immersed
in a homogeneous, compressible, inviscid fluid (air or water) in $\R^2$.
The well-posedness and stability estimate of the PML solution have been established,
but no convergence analysis of the PML method is given in \cite{Bao2018}.

In this paper, we study the time-domain PML method for the transient acoustic-elastic interaction 
problem associated with a bounded elastic body immersed in a homogeneous, compressible,
inviscid fluid (air or water) above an unbounded rough surface.
Our purpose is to introduce a time-domain PML layer to truncate the unbounded domain 
of the interaction problem above a finite layer in the $x_3$ direction containing the elastic body,
leading to a PML problem in a finite strip domain.
The idea used in \cite{chen2009,Chen2012} to construct the PML layer seems difficult to apply 
to the transient acoustic-elastic interaction problem considered in this paper. 
Motivated by \cite{wyz2019}, we make use of the real coordinate stretching technique associated 
with $[\Rt(s)]^{-1}$ in the Laplace domain with the Laplace transform variable $s\in\C_+$.
The well-posedness and stability estimate of the PML problem are then established,
by employing the Laplace transform and the energy method.
Further, we establish the error estimate between the Dirichlet-to-Neumann (DtN) operators of 
the original problem and the PML problem, which is then used to prove the exponential convergence 
of the PML method in terms of the thickness and parameters of the PML layer.

The outline of this paper is as follows. In Section \ref{sec2}, we first formulate the transient 
interaction problem and then use the exact transparent boundary condition (TBC) to reduce the  
unbounded interaction problem into an equivalent initial-boundary transmission problem in a finite 
strip domain. In addition, the well-posedness and stability are also studied for the reduced problem. 
In Section \ref{sec3}, we first propose the time-domain PML method for the acoustic-elastic interaction 
problem, based on the real coordinate stretching technique, and then establish its exponential 
convergence in terms of the thickness and parameters of the PML layer.
Conclusions are given in Section \ref{sec4}.

\section{The acoustic-elastic interaction problem}\label{sec2}
\setcounter{equation}{0}

In this section, we formulate the mathematical formulation of the interaction problem for acoustic
and elastic waves with appropriate transmission conditions on the interface between the elastic body
and the acoustic medium. In addition, an exact time-domain transparent boundary condition (TBC)
is proposed to reformulate the unbounded interaction problem into an initial-boundary value problem
in a finite strip domain. We finally establish the well-posedness and stability of solutions to
the reduced problem.

We first introduce some basic notions to be used in this paper. Throughout, let $x=(\wid{x}^T,x_3)^T$,
where $\wid{x}=(x_1,x_2)^T\in\R^2$.
Denote by $\Om$ the bounded homogeneous, isotropic elastic body with a Lipschitz boundary $\G:=\pa\Om$
immersed in the unbounded domain $\Om_f^+$, where $\Om_f^+:=\{x\in\R^3:x_3>f(\wid{x})\}$ with the
boundary $\G_f:=\pa\Om_f^+=\{x\in\R^3:x_3=f(\wid{x})\}$ described by the smooth function $f\in C^2(\R^2)$.
We assume that $\G_f$ lies between the planes $x_3=f_-$ and $x_3=f_+$,
where $f_-:=\inf_{\wid{x}\in\R^2}f(\wid{x})$ and $f_+:=\sup_{\wid{x}\in\R^2}f(\wid{x})$ are two constants.
Suppose the elastic body $\Om$ is described by a constant mass density $\rho_e>0$.
Let $\Om^c=\Om_f^+\se\ov{\Om}$ be connected and occupied by a compressible fluid with constant density $\rho_0>0$.
Define $\G_{h}:=\{\bm{x}\in\R^3: x_3=h\}$, where the positive constant $h$ is assumed to be large enough
such that $\G_h$ is over $\Om$, and let $\Om_{h}=\{x\in\R^3: f<x_3<h\}\cap\Om^c$.
See Figure \ref{geometry} for the geometric setting of the problem.
Finally, define $\C_+:=\{s=s_1+is_2\in\C: s_1,s_2\in\R\;{\rm with\;}s_1>0\}$.

\begin{figure}[!htbp]
\setcounter{subfigure}{0}
  \centering
  \includegraphics[width=4.5in]{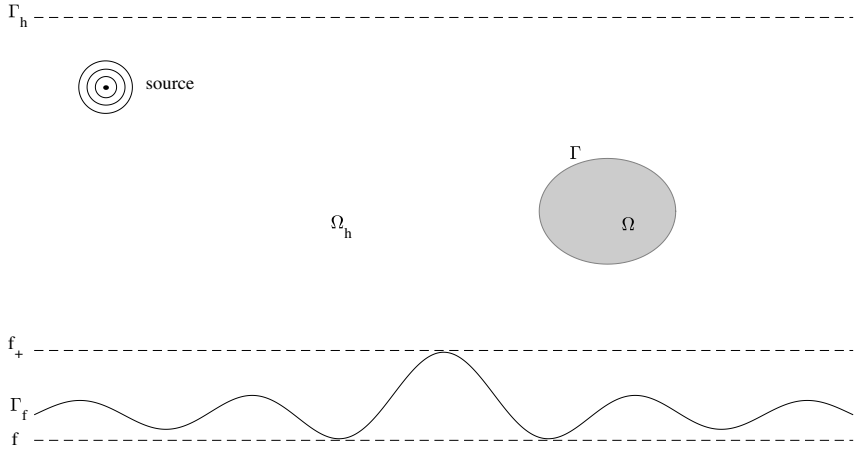}
 \caption{Geometric configuration of the interaction problem}\label{geometry}
\end{figure}

\textbf{Elastic domain}. In the elastic body $\Om$, the elastic displacement $\bm u=(u_1,u_2,u_3)^T$
is governed by the linear elastodynamic equation:
\be\label{2.1}
\rho_e\frac{\pa^{2}\bm u}{\pa t^{2}}-\Delta^*\bm u=\0\;\;\gin\;\;\Om\times(0,T)
\en
where $\Delta^*$ is the Lam\'{e} operator defined as
\ben
\Delta^*\bm u:=\mu\Delta\bm u+(\lambda+\mu)\na\dive\bm u=\dive\bm\sig(\bm u).
\enn
Here, $\bm\sig(\bm u)$ and $\bm\vep(\bm u)$ are called the stress and strain tensors, respectively,
given by
\ben
\bm\sig(\bm u)=(\la\dive\bm u)\mathbb{I}+2\mu\bm\vep(\bm u)\qquad\text{and}\qquad
\bm\vep(\bm u)=\frac{1}{2}(\na\bm u+(\na\bm u)^{T}),
\enn
where $\mathbb{I}$ is the identity matrix and $\na\bm u$ denotes the displacement gradient tensor:
\ben
\na\bm{u}=\left[
\begin{matrix}
 \partial_{x_1}u_1 & \partial_{x_2}u_1 & \partial_{x_3} u_1 \\
 \partial_{x_1}u_2 & \partial_{x_2}u_2 & \partial_{x_3} u_2 \\
 \partial_{x_1}u_3 & \partial_{x_2}u_3 & \partial_{x_3} u_3
\end{matrix}\right].
\enn
Further, Lam\'{e} constants $\lambda$ and $\mu$ are assumed to satisfy the condition
that $\mu\geq 0$ and $3\la+ 2\mu\geq 0$.

\textbf {Fluid domain}. In the unbounded fluid domain $\Om^c$, the pressure $p$ and the velocity $\bm v$
are governed by the conservation and dynamic equations in the time-domain:
\be\label{wave}
\frac{\pa p}{\pa t}=-c^2\rho_0\dive\bm v + g(x,t),\;\;\;
\frac{\pa\bm v}{\pa t}=-\rho_0^{-1}\na p\;\;\gin\;\;\Om^c\times(0,T).
\en
Eliminating the velocity $\bm v$ from (\ref{wave}), we get the wave equation for the pressure $p$:
\be\label{2.2}
\frac{\pa^{2}p}{\pa t^{2}}-c^2\Delta p=\pa_tg\quad \gin\;\; \Om^c\times(0,T),
\en
where $c$ is the sound speed and $g$ is the acoustic source which is assumed to be supported in $\Om_h$
and $g|_{t=0}=0$. We assume that $p$ satisfies the Dirichlet boundary condition on $\G_f$:
\be\label{dirichlet}
p=0\;\;\;\on\;\;\;\G_f.
\en
In addition, we impose the Upward Angular Spectrum Representation (UASR) condition on $p$
proposed in \cite{CM2005}:
\be\label{outgoing}
p(x,t)=\mathscr{L}^{-1}\Big\{\frac{1}{(2\pi)^2}\int_{\R^2}exp(i[(x_3-h)i\sqrt{s^2/c^2+|\xi|^2}
  +\wid{x}\cdot\xi])\hat{\ch{p}}(\xi,h)d\xi\Big\}
\en
for $x\in\Om_h^+:=\{x\in\R^3:x_3>h\}$, where $\mathscr{L}^{-1}$ is the inverse Laplace transform,
$\hat{\ch{p}}(\xi,h)=\mathscr{F}\ch{p}|_{\G_h}$ denotes the Fourier transform of $\ch{p}=\mathscr{L}(p)$
(the Laplace transform of $p$ with respect to $t$) restricted on $\G_h$ (the definition and relationship
of the Fourier and Laplace transforms are given in Appendix A), $\beta(\xi)=\sqrt{s^2/c^2+|\xi|^2}$
with $\Rt[\beta(\xi)]>0$ and $s\in\C_+$.

Further, we have the following transmission conditions on the interface between the elastic and fluid media
(see \cite{Hsiao2015}):

(i) The kinematic interface condition
\be\label{2.5}
\pa_n p=-\rho_0\n\cdot\pa_t^2\bm u\;\;\;\mbox{on}\;\;\G,
\en

(ii) The dynamic interface condition
\be\label{2.6}
-p\n=\bm{\sig}(\bm{u})\n\;\;\;\mbox{on}\;\;\G,
\en
where $\n$ is the unit normal on $\G$ directed into the exterior of the domain $\Om$.

To be more precise, the acoustic-elastic interaction problem we consider is that a time-dependent acoustic
wave propagates in a fluid domain above a rough surface in which a bounded elastic body is immersed.
The problem is to determine the scattered pressure in the fluid domain and the displacement field in the
elastic domain at any time. The time-dependent scattering problem can be now modelled by combining (\ref{2.1})
for the elastic displacement field $\bm{u}$ and (\ref{2.2}) for the pressure field $p$ together with
the transmission conditions (\ref{2.5})-(\ref{2.6}), the Dirichlet boundary conditions (\ref{dirichlet})
on $\G_f$ as well as the homogeneous initial conditions
\be\label{2.7}
\bm{u}(x,0)=\pa_t\bm{u}(x,0)=\0,\;\;x\in\Om,\quad\;\;p(x,0)=\pa_tp(x,0)=0,\;\; x\in\Om^c,
\en
which can be formulated mathematically as follows:
\be\label{2.8}
\begin{cases}
 \ds\rho_e\frac{\pa^{2}\bm u}{\pa t^{2}}-\Delta^*\bm u=\0& \gin\;\;\Om\times(0,T),\\
 \ds\frac{\pa^{2}p}{\pa t^{2}}-c^2\Delta p=\pa_t g & \gin\;\;\Om^c\times(0,T),\\
 \ds\textbf u(x,0)=\pa_t\bm{u}(x,0)=\textbf{0} & \gin\;\;\Om,\\
 \ds p(x,0)=\pa_tp(x,0)=0 & \gin\;\;\Om^c,\\
 \ds\pa_n p=-\rho_0\n\cdot\pa_t^2\bm u & \on\;\;\G\times(0,T), \\
 \ds-p\n=\bm{\sig}(\bm{u})\n & \on\;\;\G\times(0,T),\\
 \ds p=0 & \on\;\;\G_f\times(0,T),\\
  p\;\;\mbox{satisfies the UASR condition (\ref{outgoing}).}&
\end{cases}
\en

To study the well-posedness of the scattering problem (\ref{2.8}), we reformulate it into
a transmission problem in the strip domain $\Om_h\cup\Om$ by using the transparent boundary
condition (TBC) on the plane $\G_h$ proposed in \cite{GLZ2017}:
\be\label{2.9}
\pa_n p=\mathscr{T}[p]\quad \on\;\;\G_h\times(0,T).
\en
Then (\ref{2.8}) can be equivalently reduced to the transmission problem (TP) in $\Om_h\cup\Om$:
\be\label{reduced}
\begin{cases}
 \ds\rho_e\frac{\pa^{2}\bm u}{\pa t^{2}}-\Delta^*\bm u=\0& \gin\;\;\Om\times(0,T),\\
 \ds\frac{\pa^{2}p}{\pa t^{2}}-c^2\Delta p=\pa_t g & \gin\;\;\Om_h\times(0,T),\\
 \ds\textbf u(x,0)=\pa_t\bm{u}(x,0)=\textbf{0} & \gin\;\;\Om,\\
 \ds p(x,0)=\pa_tp(x,0)=0 & \gin\;\;\Om_h,\\
 \ds\pa_n p=-\rho_0\n\cdot\pa_t^2\bm u & \on\;\;\G\times(0,T), \\
 \ds-p\n=\bm{\sig}(\bm{u})\n & \on\;\;\G\times(0,T),\\
 \ds p=0 & \on\;\;\G_f\times(0,T),\\
 \ds\pa_n p=\mathscr{T}[p] & \on\;\;\G_h\times(0,T).
\end{cases}
\en

In the remaining part of this section, we establish the well-posedness and stability of the reduced
problem (\ref{reduced}) by using the Laplace transform.
The proof is similar to that used in \cite{GLZ2017}, and so we only present the main results
without detailed proofs. To this end, we take the Laplace transform of $p(x,t)$ and $\bm u(x,t)$,
respectively, in (\ref{reduced}) with respect to $t$ and write $\ch{p}(x,s)=\mathscr{L}(p)(x,s)$
$\ch{\bm u}(x,s)=\mathscr{L}(\bm u)(x,s)$. Then (\ref{reduced}) can be reduced to the problem in
$s$-domain:
\be\label{2.10}
\begin{cases}
 \ds\Delta^*\ch{\bm u}-\rho_es^2\ch{\bm u}=\0 & \gin\;\;\Om,\\
 \ds\Delta\ch{p}-\frac{s^{2}}{c^{2}}\ch{p}=-s\ch{g}/c^2 & \gin\;\;\Om_h,\\
 \ds\pa_n\ch{p}=-\rho_0s^2\n\cdot\ch{\bm u} & \on\;\;\G,\\
 \ds-\ch{p}\n=\bm{\sig}(\ch{\bm{u}})\n & \on\;\;\G,\\
 \ds\ch{p}=0 & \on\;\;\G_f,\\
 \ds\pa_n \ch{p}=\mathscr{B}[\ch{p}] & \on\;\;\G_h,
\end{cases}
\en
where $s\in\C_+$ and $\mathscr{B}$ is the Dirichlet-to-Neumann (DtN) operator in $s$-domain satisfying
$\mathscr{T}=\mathscr{L}^{-1}\circ\mathscr{B}\circ\mathscr{L}$.

For any function $\om(\wid{x},h)$ defined on $\G_h$, the DtN operator $\mathscr{B}$ is defined by
\be\label{2.18}
(\mathscr{B}\om)(\wid{x},h)=-\int_{\R^2}\beta(\xi)\hat{\om}(\xi,h)e^{i\xi\cdot\wid{x}}d\xi.
\en
Then the following lemma was proved in \cite{GLZ2017} (see \cite[Lemmas 2.4 and 2.5]{GLZ2017}),
where, for $s\in\R$ the space $H^s(\G_h)$ denotes the standard Sobolev space on $\G_h$ with
its norm being defined via the Fourier transform as
\be\label{4.20}
\|\phi\|^2_{H^s(\G_h)}=\int_{\R^2}(1+|\xi|^2)^s|\hat{\phi}(\xi,h)|^2d\xi,\;\;\;
\phi\in H^{s}(\G_h),\;\;s\in\R.
\en

\begin{lemma}\label{lem2.1}
$(i)$ For $s=s_1+s_2\in\C_+$ with $s_1\ge\sig_0>0$ the DtN operator $\mathscr{B}(s)$ is bounded from
$H^{1/2}(\G_{h})$ to $H^{-1/2}(\G_{h})$, that is,
\ben
\|\mathscr{B}(s)w\|_{H^{-1/2}(\G_{h})}\le C(\sig_0)|s|\|w\|_{H^{1/2}(\G_{h})}\;\;\;
\forall w\in H^{1/2}(\G_{h}),
\enn
where $C(\sig_0)$ is a constant depending only on $c$ and $\sig_0$.

$(ii)$ For any $\om\in H^{1/2}(\G_{h})$ we have
\ben
-\Rt\langle s^{-1}\mathscr{B}\om,\om\rangle_{\G_{h}}\geq 0,\;\;\;s\in\C_+,
\enn
where $\langle\cdot,\cdot\rangle_{\G_h}$ denotes the dual product on $\G_h$ between $H^{1/2}(\G_{h})$
and $H^{-1/2}(\G_{h})$.
\end{lemma}

To study the $s$-domain problem (\ref{2.10}) we introduce the Hilbert space
$H:=H_{\G_f}^1(\Om_h)\times H^1(\Om)^3$, where $H_{\G_f}^1(\Om_h):=\{u\in H^1(\Om_h):u=0\;\on\;\G_f\}$.
The norm of the product space $H$ is defined by
\be\label{3.3a}
\|(q,\bm v)\|_H:=\left[\|q\|^2_{H^1(\Om_h)}+\|\bm v\|^2_{H^1(\Om)^3}\right]^{1/2}\;\;\;
\text{for}\;\;(q,\bm{v})\in H,
\en
where $\|\cdot\|_{H^1(\Om_h)}$ denotes the usual $H^1$-norm and $\|\cdot\|_{H^1(\Om)^3}$ is defined by
\ben
\|\bm v\|_{H^1(\Om)^3}:=\left(\|\bm v\|^2_{L^2(\Om)^3}+\|\na\bm v\|^2_{F(\Om)}\right)^{1/2}
\enn
with the Frobenius norm
\ben
\|\na\bm{v}\|_{F(\Om)}:=\left(\sum_{j=1}^3\int_{\Om}|\na v_j|^2dx\right)^{1/2}.
\enn
It is easy to verify that
\ben
\|\na\bm{v}\|^2_{F(\Om)}+\|\na\cdot\bm{v}\|^2_{L^2(\Om)}\lesssim \|\bm v\|^2_{H^1(\Om)^3}.
\enn
Hereafter, the expression $a\lesssim b$ or $a\gtrsim b$ means $a\leq Cb$ or $a\geq Cb$, respectively,
for a generic positive constant $C$ which does not depend on any function and important parameters
in our model.

The variational formulation of (\ref{2.10}) can be obtained as follows: Find a solution
$(\ch{p},\ch{\bm{u}})\in H:=H_{\G_f}^1(\Om_h)\times H^1(\Om)^3$ such that
\be\label{3.2}
a\left((\ch{p},\ch{\bm{u}}),(q,\bm{v})\right)=\int_{\Om_h}\frac{\ch{g}}{c^2}\cdot\ov{q}dx,\;\;\quad
\forall\;\;(q,\bm{v})\in H,
\en
where the sesquilinear form $a(\cdot,\cdot)$ is defined as
\be\label{3.3}
a\left((\ch{p},\ch{\bm{u}}),(q,\bm{v})\right)&=&\int_{\Om_h}\left(s^{-1}\na\ch{p}\cdot\na\ov{q}
+\frac{s}{c^2}\ch{p}\cdot\ov{q}\right)dx\no\\
&+&\int_\Om\left[\rho_0\ov{s}\left(\la(\na\cdot\ch{\bm{u}})(\na\cdot\ov{\bm{v}})
 +2\mu\bm\vep(\ch{\bm{u}}):\bm\vep(\ov{\bm{v}})\right)
 +\rho_0\rho_e|s|^2s\ch{\bm{u}}\cdot\ov{\bm{v}}\right]dx\no\\
&-&\int_{\G_h}s^{-1}\mathscr{B}[\ch{p}]\cdot\ov{q}d\g
-\rho_0\int_\G s\n\cdot\ch{\bm{u}}\ov{q}d\g+\rho_0\int_\G\ov{s}\ch{p}\n\cdot\ov{\bm{v}}d\g
\en
with $A:B=tr(AB^{T})$ denoting the Frobenius inner product of the square matrices $A$ and $B$.

Letting $(q,\bm{v})=(\ch{p},\ch{\bm{u}})$ in (\ref{3.3}) and setting $\bm\om=(\ch{p},\ch{\bm u})$, applying 
the famous Korn's inequality \cite[Chapter 10]{Mclean2000}
\ben
\|\bm\vep(\bm v)\|^2_{F(\Om)}+\|\bm v\|^2_{L^2(\Om)^3}\geq C_\Om\|\bm v\|^2_{H^1(\Om)^3},\;\;\forall\;\bm v\in H^1(\Om)^3
\enn
we obtain that
\be\label{coercivity}
\Rt[a(\bm\om,\bm\om)]&\ge&\frac{s_1}{|s|^2}\left[\|\na\ch{p}\|_{L^2(\Om_h)^3}^2
+\|\frac{s}{c}\ch{p}\|_{L^2(\Om_h)}^2\right]+\rho_0s_1\left[2\mu\|\bm\vep(\ch{\bm u})\|^2_{F(\Om)}
+\rho_e\|s\ch{\bm u}\|^2_{L^2(\Om)^3}\right]\no\\
&\ge&\frac{s_1}{|s|^2}C_1\|\ch{p}\|_{H^1(\Om_h)}^2+\rho_0s_1C_2\|\ch{\bm u}\|_{H^1(\Om)^3}^2\no\\
&\ge& C(\|\ch{p}\|_{H^1(\Om_h)}^2+\|\ch{\bm u}\|_{H^1(\Om)^3}^2)\no\\
&=&C\|\bm\om\|_{H}^2,
\en
where use has been made of Lemma \ref{lem2.1} (ii) to get the first inequality,
$C=\min\{{s_1C_1}/{|s|^2},\rho_0s_1C_2\}$, $C_1=\min\{1,{|s|^2}/{c^2}\}$
and $C_2=C_\Om\min\{2\mu,\rho_e\min\{1,|s|^2\}\}$. This means that the sesquilinear form $a(\cdot,\cdot)$ is
uniformly coercive in $H$.
By Lemma \ref{lem2.1} (i), the trace theorem (see \cite[Lemma 2.2]{GLZ2017})
and the Lax-Milgram theorem, we can obtain the following result on the well-posedness of
the $s$-domain problem (\ref{2.10}) or equivalently its variational formulation (\ref{3.3}).

\begin{lemma}\label{thm3.1}
For each $s\in\C_+$, the variational problem $(\ref{3.2})$ has a unique solution
$(\ch{p},\ch{\bm{u}})\in H$ satisfying that
\be\label{3.4}
&&\|\na\ch{p}\|_{L^2(\Om_h)^3}+\|s\ch{p}\|_{L^2(\Om_h)}\lesssim\frac{|s|}{s_1}\|\ch{g}\|_{L^2(\Om_h)},\\
&&\|\na\ch{\bm{u}}\|_{F(\Om)}+\|\na\cdot\ch{\bm{u}}\|_{L^2(\Om)}+\|s\ch{\bm{u}}\|_{{L^2(\Om)}^3}
\lesssim\frac{1}{s_1\min\{1,s_1\}}\|\ch{g}\|_{L^2(\Om_h)}.\label{3.5}
\en
\end{lemma}

To prove the well-posedness of the reduced problem (\ref{reduced}),
and establish the convergence of the PML method,
we need the following assumptions on the inhomogeneous term $g$:
\be\label{assumption}
g\in H^3(0,T;L^2(\Om_h)),\;\;g|_{t=0}=\pa_tg|_{t=0}=\pa_t^2g|_{t=0}=0.
\en
Further, we always assume that $g$ can be extended to $\infty$ with respect to $t$ such that
\be\label{assumption1}
g\in H^3(0,\infty;L^2(\Om_h)),\;\;\|g\|_{H^3(0,\infty;L^2(\Om_h))}\lesssim\|g\|_{H^3(0,T;L^2(\Om_h))}.
\en

By using Lemma \ref{thm3.1} and a similar argument as in the proof of Theorem 3.2 in \cite{GLZ2017},
the well-posedness and stability of the reduced problem (\ref{reduced}) can be obtained.

\begin{theorem}\label{thm3.3}
The reduced problem $(\ref{reduced})$ has a unique solution $(p,\bm{u})$ such that
\ben
&& p\in L^2\left(0,T;H_{\G_f}^1(\Om_h)\right)\cap H^1\left(0,T;L^2(\Om_h)\right),\\
&&\bm{u}\in L^2\left(0,T;H^1(\Om)^3\right)\cap H^1\left(0,T;L^2(\Om)^3\right)
\enn
with the estimates
\be\label{3.1a}
&&\max\limits_{t\in[0,T]}\left[\|\pa_t p\|_{L^2(\Om_h)}+\|\na p\|_{L^2(\Om_h)^3}\right]
\lesssim\|\pa_t g\|_{L^1(0,T;L^2(\Om_h))},\\ \label{3.1b}
&&\max\limits_{t\in[0,T]}\left[\|\pa_t\bm{u}\|_{L^2(\Om)^3}+\|\na\cdot\bm{u}\|_{L^2(\Om)}
+\|\na\bm{u}\|_{F(\Om)}\right]\lesssim\|\pa_t g\|_{L^1(0,T;L^2(\Om_h))}.
\en
\end{theorem}

\section{The time-domain PML problem}\label{sec3}
\setcounter{equation}{0}

In this section, we shall derive the time-domain PML formulation for the acoustic-elastic interaction
problem (\ref{2.8}). The well-posedness and stability of the PML problem can be established based on
a similar method as used in the proof of Theorem 3.2 in \cite{GLZ2017}.
Finally, we prove the exponential convergence of the time-domain PML method via constructing a
special PML layer in the $x_3$-direction, based on the real coordinate stretching technique.

\subsection{The PML problem and its well-posedness}

Let us first introduce the PML geometry which is presented in Figure \ref{PML}.
Let $\Om_{h+L}=\{\bm{x}\in\R^3: f<x_3<h+L\}\cap\Om^c$ denote the truncated PML domain
and let $\Om_{h}^L=\{\bm{x}\in\R^3: h<x_3<h+L\}$ denote the PML layer with the exterior
boundary $\G_{h+L}:=\{\bm{x}\in\R^3: x_3=h+L\}$, where $L>0$ is the thickness of the PML layer.
Now, let $s_1> 0$ be an arbitrarily fixed parameter and let us introduce the PML medium property $\sig=\sig(x_3)$:
\be\label{4.1}
\sig(x_3)=\begin{cases}
  \ds  1 &\quad \text{if}\;\;x_3\leq h,\\
  \ds  1+s_1^{-1}\sig_0(\frac{x_3-h}{L})^m&\quad \text{if}\;\;h<x_3\leq h+L,
  \end{cases}
\en
where $\sig_0$ is a positive constant, $m\geq 1$ is a given integer. 
In what follows, we will take the real part of the Laplace transform variable $s\in\C_+$ to be
$s_1$, that is, $\Rt(s)=s_1$.

\begin{figure}[!htbp]
\setcounter{subfigure}{0}
  \centering
  \includegraphics[width=4.2in]{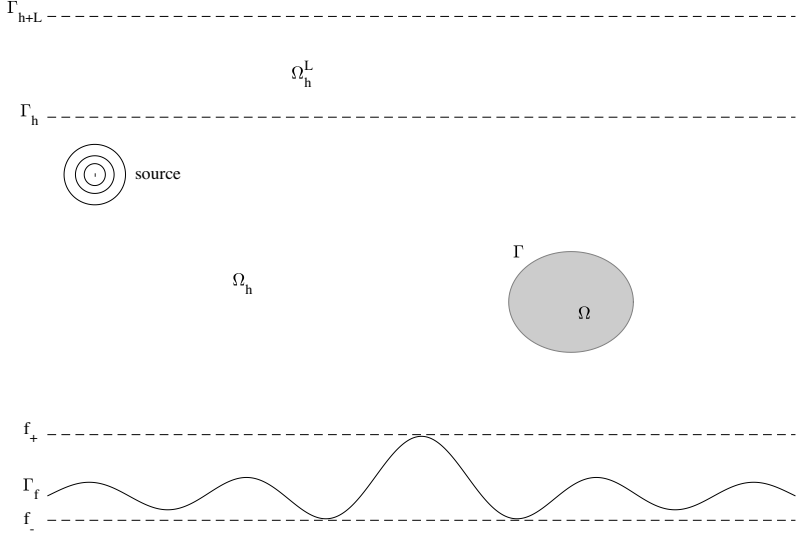}
 \caption{Geometric configuration of the truncated PML problem}\label{PML}
\end{figure}

We now derive the PML equation by the technique of change of variables, starting with
the real stretched coordinate $\hat{x}=(\hat{x}_1,\hat{x}_2,\hat{x}_3)$ with
\ben
\hat{x}_1=x_1,\;\; \hat{x}_2=x_2,\;\; \hat{x}_3=\int_{f_-}^{x_3}\sig(\tau)d\tau+f_-.
\enn

Taking the Laplace transform of the wave equation (\ref{2.2}) with respect to $t$ gives
\be\label{4.2}
\Delta\ch{p}-\frac{s^{2}}{c^{2}}\ch{p}=0\;\;\;\gin\;\;\; \Om_{h}^L.
\en
Denote by $\ch{p}_{pml}$ the PML extension of the pressure $\ch{p}$ satisfying (\ref{4.2}).
Formally, the technique of change of variables requires $\ch{p}_{pml}$ to satisfy
\ben
\sum_{j=1}^{3}\frac{\pa^2\ch{p}_{pml}}{\pa\hat{x}_j^2}-\frac{s^{2}}{c^{2}}\ch{p}_{pml}=0
\;\;\;\gin\;\;\;\Om_{h}^L.
\enn
Then, by the chain rule and using the fact that $d\hat{x}_3/dx_3=\sig$, we obtain the PML equation
\be\label{4.3}
\Delta_p\ch{p}_{pml}-\frac{s^{2}\sig}{c^{2}}\ch{p}_{pml}=0\;\;\;\gin\;\;\;\Om_{h}^L,
\en
where
\ben
\Delta_p:=\sum_{j=1}^{2}\frac{\pa}{\pa x_j}\left(\sig\frac{\pa}{\pa x_j}\right)
+\frac{\pa}{\pa x_3}\left(\frac{1}{\sig}\frac{\pa}{\pa x_3}\right)=\na\cdot(\mathbb{D}\na)
\enn
with the diagonal matrix $\mathbb{D}=\diag(\sig,\sig,1/\sig)$.

Combining the elastic wave equation (\ref{2.1}) and the interface conditions (\ref{2.5})-(\ref{2.6}),
we obtain the truncated PML problem in $s$-domain:
\begin{subnumcases}{} \label{4.5a}
\ds\Delta^*\ch{\bm u}_{pml}-\rho_es^2\ch{\bm u}_{pml}=\0 & $\gin\;\;\Om$,\\ \label{4.5b}
\ds\Delta_p\ch{p}_{pml}-\frac{s^{2}\sig}{c^{2}}\ch{p}_{pml}=-s\ch{g}/c^2 & $\gin\;\;\Om_{h+L}$,\\ \label{4.5c}
\ds\pa_n\ch{p}_{pml}=-\rho_0s^2\n\cdot\ch{\bm u}_{pml}& $\on\;\;\G$,\\\label{4.5d}
\ds -\ch{p}_{pml}\n=\bm{\sig}(\ch{\bm u}_{pml})\n & $\on\;\;\G$,\\ \label{4.5e}
\ds \ch{p}_{pml}=0& $\on\;\;\G_f$,\\\label{4.5f}
\ds \ch{p}_{pml}=0& $\on\;\;\G_{h+L}$,
\end{subnumcases}
where the unbounded domain is truncated into the finite strip layer $\Om_{h+L}$ by imposing
the homogeneous Dirichlet boundary condition on $\G_{h+L}$,
in view of the exponential decay of the transformed pressure field $\ch{p}$.


We now prove the well-posedness of the truncated PML problem (\ref{4.5a})-(\ref{4.5f})
by the variational method in the Hilbert space $\wid{H}:=H_0^1(\Om_{h+L})\times H^1(\Om)^3$,
where $H_0^1(\Om_{h+L}):=\{u\in H^1(\Om_{h+L}):u=0\;\on\;\G_f\cup\G_{h+L}\}$ and
the norm of $\wid{H}$ is defined similarly as that of $H$ in (\ref{3.3a})
with $\Om_h$ replaced by $\Om_{h+L}$.
To this end, use Green's and Betti's formulas as well as the transmission conditions
(\ref{4.5c})-(\ref{4.5d}) to obtain the following variational formulation of the PML problem
(\ref{4.5a})-(\ref{4.5f}): find a solution $(\ch{p}_{pml},\ch{\bm{u}}_{pml})\in\wid{H}$ such that
\be\label{4.6}
a_{pml}\big((\ch{p}_{pml},\ch{\bm{u}}_{pml}),(q,\bm{v})\big)
=\int_{\Om_{h}}\frac{\ch{g}}{c^2}\cdot\ov{q}dx\;\;\quad\forall\;(q,\bm{v})\in\wid{H},
\en
where the sesquilinear form $a_{pml}(\cdot,\cdot)$ is defined as
\ben
&&a_{pml}\left((\ch{p}_{pml},\ch{\bm{u}}_{pml}),(q,\bm{v})\right)\\
&&\quad=\int_{\Om_{h+L}}(s^{-1}\mathbb{D}\na\ch{p}_{pml}\cdot\na\ov{q}
+\frac{s\sig}{c^2}\ch{p}_{pml}\cdot\ov{q})dx\\
&&\qquad+\int_\Om\left[\rho_0\ov{s}\left[\la(\na\cdot\ch{\bm{u}}_{pml})(\na\cdot\ov{\bm{v}})
    +2\mu\bm\vep(\ch{\bm{u}}_{pml}):\bm\vep(\ov{\bm{v}})\right]
+\rho_0\rho_e|s|^2s\ch{\bm{u}}_{pml}\cdot\ov{\bm{v}}\right]dx\\
&&\qquad-\rho_0\int_\G s\n\cdot\ch{\bm{u}}_{pml}\ov{q}d\g
+\rho_0\int_\G\ov{s}\ch{p}_{pml}\n\cdot\ov{\bm{v}}d\g.
\enn
Noting that $1\leq\sig\leq1+s_1^{-1}\sig_0$ for $x\in\Om_{h+L}$, combining the Korn's inequality, we have
\ben
&&\Rt\left[a_{pml}\left((\ch{p}_{pml},\ch{\bm{u}}_{pml}),(\ch{p}_{pml},\ch{\bm{u}}_{pml})\right)\right]\\
&&=\Rt\int_{\Om_{h+L}}s^{-1}\left(\sig|\pa_{x_1}\ch{p}_{pml}|^2+\sig|\pa_{x_2}\ch{p}_{pml}|^2
+\frac{1}{\sig}|\pa_{x_3}\ch{p}_{pml}|^2\right)dx
+\int_{\Om_{h+L}}\frac{s_1\sig}{c^2}|\ch{p}_{pml}|^2dx\\
&&\qquad+\rho_0s_1\left(\la\|\na\cdot\ch{\bm{u}}_{pml}\|_{{L^2(\Om)}}^2
    +2\mu\|\bm\vep(\ch{\bm{u}}_{pml})\|_{F(\Om)}^2
+\rho_e\|s\ch{\bm{u}}_{pml}\|_{{L^2(\Om)}^3}^2\right)\\
&&\gtrsim\frac{1}{1+s_1^{-1}\sig_0}\frac{s_1}{|s|^2}\left(\|\na\ch{p}_{pml}\|_{L^2(\Om_{h+L})^3}^2
+\|s\ch{p}_{pml}\|_{L^2(\Om_{h+L})}^2\right)\\
&&\qquad+s_1\min\{1,s_1^2\}\left(\|\na\ch{\bm{u}}_{pml}\|_{F(\Om)}^2+\|\na\cdot\ch{\bm{u}}_{pml}\|_{{L^2(\Om)}}^2
+\|s\ch{\bm{u}}_{pml}\|_{{L^2(\Om)}^3}^2\right),
\enn
which means that $a_{pml}(\cdot,\cdot)$ is uniformly coercive in $\wid{H}$.

Arguing similarly as in the proof of Lemma \ref{thm3.1} (noting that the TBC in the $s$-domain is now
replaced with the Dirichlet boundary condition), we can obtain the following theorem.

\begin{theorem}\label{thm4.1}
 The truncated PML variational problem (\ref{4.6}) has a unique solution 
$(\ch{p}_{pml},\ch{\bm{u}}_{pml})\in\wid{H}$ for each $s\in\C_+$
 with $\Rt(s)=s_1>0$. Further, we have the following estimates
\be\label{4.9}
&&\|\na\ch{p}_{pml}\|_{L^2(\Om_{h+L})^3}+\|s\ch{p}_{pml}\|_{L^2(\Om_{h+L})}
\lesssim\frac{(1+s_1^{-1}\sig_0)|s|}{s_1}\|\ch{ g}\|_{L^2(\Om_h)},\\ \label{4.10}
&&\|\na\ch{\bm{u}}_{pml}\|_{F(\Om)}+\|\na\cdot\ch{\bm{u}}_{pml}\|_{L^2(\Om)}
+\|s\ch{\bm{u}}_{pml}\|_{{L^2(\Om)}^3}
\lesssim\frac{\sqrt{1+s_1^{-1}\sig_0}}{s_1\min\{1,s_1\}}\|\ch{g}\|_{L^2(\Om_h)}.\;\;
\en
\end{theorem}

Taking the inverse Laplace transform of the system (\ref{4.5a})-(\ref{4.5f}),
we obtain the truncated PML problem in the time-domain:
\be\label{tdpml}
\begin{cases}
\ds\Delta^*\bm u_{pml}-\rho_e\pa_t^{2}\bm u_{pml}=\0 & \gin\;\;\Om\times(0,T),\\
\ds\Delta_p p_{pml}-\frac{\sig}{c^2}\pa_t^{2}p_{pml}=-\pa_t g/c^2 & \gin\;\;\Om_{h+L}\times(0,T),\\
\ds\bm u_{pml}(x,0)=\pa_t\bm u_{pml}(x,0)=\0 & \gin\;\;\Om,\\
\ds p_{pml}(x,0)=\pa_t p_{pml}(x,0)=0 & \gin\;\;\Om_{h+L},\\
\ds\pa_n p_{pml}=-\rho_0\n\cdot\pa_t^2\bm u_{pml} & \on\;\;\G\times(0,T),\\
\ds-p_{pml}\n=\bm{\sig}(\bm u_{pml})\n & \on\;\;\G\times(0,T),\\
\ds p_{pml}=0 & \on\;\;\G_f\times(0,T),\\
\ds p_{pml}=0 & \on\;\;\G_{h+L}\times(0,T).
\end{cases}
\en
Note that $s_1$ appearing in PML mefium property $\sig$ is an arbitrarily fixed, positive parameter,
as mentioned earlier at the beginning of this subsection. In the Laplace transform domain,
the transform variable $s\in\C_+$ is taken so that $\Rt(s)=s_1>0$, and in the subsequent study of
the well-posedness and convergence of the truncated PML problem (\ref{tdpml}), we take $s_1=1/T$.

By using Theorem \ref{thm4.1} and a similar argument as in the proof of Theorem 3.2 in \cite{GLZ2017},
we can establish the following result on the well-posedness and stability of the PML problem (\ref{tdpml}).

\begin{theorem}\label{thm4.1+}
Let $s_1=1/T$. Then the truncated PML problem $(\ref{tdpml})$ has a unique solution $\left(p_{pml},\bm{u}_{pml}\right)$
with
\ben
&& p_{pml}\in L^2\left(0,T;H^1_0(\Om_{h+L})\right)\cap H^1\left(0,T;L^2(\Om_{h+L})\right),\\
&& \bm{u}_{pml}\in L^2\left(0,T;H^1(\Om)^3\right)\cap H^1\left(0,T;L^2(\Om)^3\right)
\enn
and satisfies the stability estimates
\ben
&&\max\limits_{t\in[0,T]}(\|\pa_t p_{pml}\|_{L^2(\Om_{h+L})}
+\|\na p_{pml}\|_{L^2(\Om_{h+L})^3})\lesssim (1+\sig_0T)\|\pa_t g\|_{L^1(0,T;L^2(\Om_h))},\\
&&\max\limits_{t\in[0,T]}(\|\pa_t\bm{u}_{pml}\|_{L^2(\Om)^3}+\|\na\cdot\bm{u}_{pml}\|_{L^2(\Om)}
+\|\na\bm{u}_{pml}\|_{F(\Om)})\\
&&\qquad\qquad\qquad\;\;\lesssim\sqrt{1+\sig_0T}\|\pa_t g\|_{L^1(0,T;L^2(\Om_h))}.
\enn
\end{theorem}

\subsection{A DtN operator for the PML problem}

We now derive an error estimate between the DtN operators of the original scattering problem
(\ref{2.8}) or equivalently the problem (\ref{reduced}) and the PML problem (\ref{tdpml}).
We start by introducing the DtN operator $\mathscr{B}_{pml}:H^{1/2}(\G_h)\ra H^{-1/2}(\G_h)$
associated with the truncated PML problem (\ref{4.5a})-(\ref{4.5f}).
Given $\phi\in H^{1/2}(\G_h)$, define $\mathscr{B}_{pml}\phi:=\pa_{x_3}v\;\on\;\G_h$,
where $v\in H^{1}(\Om_h^L)$ satisfies the following problem in the PML layer:
\be\label{4.13}
 \begin{cases}
  \ds\Delta_pv-\frac{s^{2}\sig}{c^{2}}v=0 & \gin\;\;\Om_{h}^L,\\
  \ds v=\phi & \on\;\;\G_{h},\\
  \ds v=0 & \on\;\;\G_{h+L}.
 \end{cases}
\en
Taking the Fourier transform of (\ref{4.13}) with respect to $\tilde{x}$, we get
\be\label{4.14}
 \begin{cases}
  \ds\frac{\pa}{\pa_{x_3}}\left(\frac{1}{\sig}\frac{\pa}{\pa x_3}\hat{v}(\xi,x_3)\right)
  \ds-\left(\frac{s^{2}\sig}{c^{2}}+\sig|\xi|^2\right)\hat{v}(\xi,x_3)=0,& h<x_3<h+L,\\
  \ds\hat{v}(\xi,x_3)=\hat{\phi}(\xi,h),& x_3=h,\\
  \ds\hat{v}(\xi,x_3)=0, & x_3=h+L.
 \end{cases}\;\;\;
\en
Solving (\ref{4.14}) gives
\be\label{4.15}
\hat{v}(\xi,x_3)=Ae^{\beta(\xi)(\hat{x}_3-h)}+Be^{-\beta(\xi)(\hat{x}_3-h)},\;\;\;h<x_3<h+L,
\en
where $\beta(\xi)=\sqrt{s^2/c^2+|\xi|^2}$ with $\Rt[\beta(\xi)]>0$ and $A,B$ are unknown
functions of $\xi$ to be determined.

Take $x_3=h$. Then, by the definition of $\sig$ we have $\hat{x}_3=h$, and so (\ref{4.15})
implies that
\be\label{AB1}
A+B=\hat{\phi}(\xi,h).
\en
Now, take $x_3=h+L$. Then a direct calculation gives
\ben\label{hight}
\hat{x}_3-h=\int_{f_-}^{h+L}\sig(\tau)d\tau+f_--h
=\int_{h}^{h+L}\sig(\tau)d\tau=\left(1+\frac{s_1^{-1}\sig_0}{m+1}\right)L:=\wid{L}.\;\quad
\enn
This, together with (\ref{4.15}), implies that
\be\label{AB2}
Ae^{\beta(\xi)\wid{L}}+Be^{-\beta(\xi)\wid{L}}=0.
\en
Solving (\ref{AB1}) and (\ref{AB2}) for $A$ and $B$ gives
\ben
A=-\frac{e^{-\beta(\xi)\wid{L}}\hat{\phi}(\xi,h)}{e^{\beta(\xi)\wid{L}}-e^{-\beta(\xi)\wid{L}}},\;\;\;
B=\frac{e^{\beta(\xi)\wid{L}}\hat{\phi}(\xi,h)}{e^{\beta(\xi)\wid{L}}-e^{-\beta(\xi)\wid{L}}}.
\enn
Then we obtain the following solution of (\ref{4.14}):
\be\label{4.16}
\hat{v}(\xi,x_3)=\frac{e^{-\beta(\xi)(\hat{x}_3-h-\wid{L})}-e^{\beta(\xi)(\hat{x}_3-h-\wid{L})}}
{e^{\beta(\xi)\wid{L}}-e^{-\beta(\xi)\wid{L}}}\hat{\phi}(\xi,h),
\;\;\;h<x_3<h+L.
\en
Taking the derivative of (\ref{4.16}) with respect to $x_3$ and evaluating its value at $x_3=h$,
we obtain
\ben
\frac{\pa\hat{v}(\xi,h)}{\pa x_3}=-\beta(\xi)\frac{e^{-\beta(\xi)\wid{L}}
+e^{\beta(\xi)\wid{L}}}{e^{\beta(\xi)\wid{L}}-e^{-\beta(\xi)\wid{L}}}\hat{\phi}(\xi,h),
\enn
where we have used the fact that $\sig(h)=1$.
Now define
\ben
\wi{\mathscr{B}_{pml}\phi}=-\beta(\xi)\frac{e^{-\beta(\xi)\wid{L}}
 +e^{\beta(\xi)\wid{L}}}{e^{\beta(\xi)\wid{L}}-e^{-\beta(\xi)\wid{L}}}\hat{\phi}(\xi,h).
\enn
Then taking the inverse Fourier transform of the above equation leads to
the DtN operator $\mathscr{B}_{pml}$ defined as follows:
\ben
\mathscr{B}_{pml}\phi(\wid{x},h)=-\int_{\R^2}\beta(\xi)\frac{e^{-\beta(\xi)\wid{L}}
+e^{\beta(\xi)\wid{L}}}{e^{\beta(\xi)\wid{L}}-e^{-\beta(\xi)\wid{L}}}
\hat{\phi}(\xi,h)e^{i\xi\cdot\wid{x}}d\xi,\;\;\;\phi\in H^{1/2}(\G_h).
\enn
Hence, the truncated PML problem (\ref{4.5a})-(\ref{4.5f}) can be equivalently reduced to
the boundary value problem in $\Om_h\cup\Om$:
\be\label{4.17}
\begin{cases}
\ds \Delta^*\ch{\bm u}_{pml}-\rho_es^2\ch{\bm u}_{pml}=\0& \gin\;\;\;\Om, \\
\ds \Delta\ch{p}_{pml}-\frac{s^{2}}{c^{2}}\ch{p}_{pml}=-s\ch{g}/c^2& \gin\;\;\;\Om_{h},\\
\ds \pa_n\ch{p}_{pml}=-\rho_0s^2\n\cdot\ch{\bm u}_{pml}& \on\;\;\;\G,\\
\ds -\ch{p}_{pml}\n=\bm{\sig}(\ch{\bm u}_{pml})\n& \on\;\;\;\G,\\
\ds \ch{p}_{pml}=0& \on\;\;\;\G_f,\\
\ds \pa_{x_3}\ch{p}_{pml}=\mathscr{B}_{pml}[\ch{p}_{pml}]& \on\;\;\;\G_{h}.
\end{cases}
\en

Similarly as in the derivation of the problems (\ref{3.2}) and (\ref{4.6}), we can obtain
the variational formulation of the problem (\ref{4.17}):
find $(\ch{p}_{pml},\ch{\bm u}_{pml})\in H$ such that
\be\label{4.18}
a_{p}\big((\ch{p}_{pml},\ch{\bm{u}}_{pml}),(q,\bm{v})\big)
=\int_{\Om_{h}}\frac{\ch{g}}{c^2}\cdot\ov{q}dx\quad\;\;\forall\;(q,\bm{v})\in H,
\en
where the sesquilinear form $a_p(\cdot,\cdot)$ is defined as
\ben
&&a_{p}\big((\ch{p}_{pml},\ch{\bm{u}}_{pml}),(q,\bm{v})\big)\\
&&\;\;=\int_{\Om_{h}}(s^{-1}\na\ch{p}_{pml}\cdot\na\ov{q}
+\frac{s}{c^2}\ch{p}_{pml}\cdot\ov{q})dx\\
&&\;\quad+\int_\Om\Big[\rho_0\ov{s}\big(\la(\na\cdot\ch{\bm{u}}_{pml})(\na\cdot\ov{\bm{v}})+
2\mu\bm\vep(\ch{\bm{u}}_{pml}):\bm\vep(\ov{\bm{v}}))\big)
+\rho_0\rho_e|s|^2s\ch{\bm{u}}_{pml}\cdot\ov{\bm{v}}\Big]dx\\
&&\;\quad-\int_{\G_h}s^{-1}\mathscr{B}_{pml}[\ch{p}_{pml}]\cdot\ov{q}d\g
-\rho_0\int_\G s\n\cdot\ch{\bm{u}}_{pml}\ov{q}d\g
+\rho_0\int_\G\ov{s}\ch{p}_{pml}\n\cdot\ov{\bm{v}}d\g.
\enn

\subsection{Exponential convergence of the time-domain PML solution}

In this subsection, we derive an error estimate between the solution $(p,\bm u)$ of
the original problem (\ref{2.8}) and the solution $(p_{pml},\bm u_{pml})$ of the PML problem (\ref{tdpml}).
To this end, we introduce some notations and norms. Denote by $L(X,Y)$ the standard space of
the bounded linear operators from the Hilbert space $X$ to the Hilbert space $Y$.

We now establish the following error estimate between the DtN operators $\mathscr{B}$ and
$\mathscr{B}_{pml}$ which is essential for the convergence analysis of the PML method.

\begin{theorem}\label{thm4.3}
Let $\ov{L}={\sig_0L}/({m+1})$, $s=s_1+is_2$ with $s_1>0$. Then we have
\be\label{thm4.3a}
\|\mathscr{B}-\mathscr{B}_{pml}\|_{L(H^{1/2}(\G_h),H^{-1/2}(\G_h))}\le\max\left\{1,\frac{|s|}{c}\right\}
\frac{2e^{-2\ov{L}/c}}{1-e^{-2\ov{L}/c}}:=C_{U}(s,\ov{L}).\;\qquad
\en
\end{theorem}

\begin{proof}
By the definition of the norm (\ref{4.20}) it follows that for any $\phi\in H^{1/2}(\G_h)$,
\ben
&&\|(\mathscr{B}-\mathscr{B}_{pml})\phi\|_{H^{-1/2}(\G_h)}^2\\
&&\quad=\int_{\R^2}|\beta(\xi)|^2(1+|\xi|^2)^{-1/2}\left|1-\frac{e^{-\beta(\xi)\wid{L}}
+e^{\beta(\xi)\wid{L}}}{e^{\beta(\xi)\wid{L}}-e^{-\beta(\xi)\wid{L}}}\right|^2|\hat{\phi}(\xi)|^2d\xi\\
&&\quad\leq\int_{\R^2}\left(\frac{|s|^2}{c^2}+|\xi|^2\right)(1+|\xi|^2)^{-1/2}\left|1-\frac{e^{-\beta(\xi)\wid{L}}
+e^{\beta(\xi)\wid{L}}}{e^{\beta(\xi)\wid{L}}-e^{-\beta(\xi)\wid{L}}}|^2|\hat{\phi}(\xi)\right|^2d\xi\\
&&\quad\le\max\left\{1,\frac{|s|^2}{c^2}\right\}\sup\limits_{\xi\in\R^2}\left|1-\frac{e^{-\beta(\xi)\wid{L}}
+e^{\beta(\xi)\wid{L}}}{e^{\beta(\xi)\wid{L}}-e^{-\beta(\xi)\wid{L}}}\right|^2\|\phi\|^2_{H^{1/2}(\G_h)}.
\enn
Thus
\ben
\|\mathscr{B}-\mathscr{B}_{pml}\|_{L(H^{1/2}(\G_h),H^{-1/2}(\G_h))}\le\max\left\{1,\frac{|s|}{c}\right\}
\sup\limits_{\xi\in\R^2}\left|1-\frac{e^{-\beta(\xi)\wid{L}}+e^{\beta(\xi)\wid{L}}}
{e^{\beta(\xi)\wid{L}}-e^{-\beta(\xi)\wid{L}}}\right|
\enn
It is easy to see that
\be\label{4.24}
\sup\limits_{\xi\in\R^2}\left|1-\frac{e^{-\beta(\xi)\wid{L}}+e^{\beta(\xi)\wid{L}}}
{e^{\beta(\xi)\wid{L}}-e^{-\beta(\xi)\wid{L}}}\right|
&=&\sup\limits_{\xi\in\R^2}\frac{\left|2e^{-2\beta_r(\xi)\wid{L}}\right|}
{\left|1-e^{-2[\beta_r(\xi)+i\beta_i(\xi)]\wid{L}}\right|}\no\\
&\le&\sup\limits_{\xi\in\R^2}\frac{2e^{-2\beta_r(\xi)\wid{L}}}{1-e^{-2\beta_r(\xi)\wid{L}}},
\en
where $\beta_r(\xi)=\Rt[\beta(\xi)]$ and $\beta_i(\xi)=\I[\beta(\xi)]$.
By the formulas
\ben
z^{1/2}=\sqrt{\frac{|z|+z_1}{2}}+i\,\mbox{sign}(z_2)\sqrt{\frac{|z|-z_1}{2}},\;\;\;
z=z_1+iz_2,\;\;\;\Rt[z^{1/2}]>0,
\enn
we have
\ben
\beta_r(\xi)&=&\frac1{\sqrt{2}}\left(|\beta^2(\xi)|+\Rt[\beta^2(\xi)]\right)^{1/2}\\
&=&\frac1{\sqrt{2}}\left(\left[\left(\frac{s_1^2-s_2^2}{c^2}+|\xi|^2\right)^2
+\frac{4s_1^2s_2^2}{c^4}\right]^{1/2}+\frac{s_1^2-s_2^2}{c^2}+|\xi|^2\right)^{1/2}.
\enn
Since ${2e^{-2\beta_r(\xi)\wid{L}}}/[{1-e^{-2\beta_r(\xi)\wid{L}}}]$ is monotonically decreasing
with respect to $\beta_r(\xi)$, then we need to seek the minimum of $\beta_r(\xi)$ in $\R^2$.
A direct calculation yields that $\xi=0$ is the unique minimum point of the function $\beta_r(\xi)$,
and thus
\ben
\beta_r(0)=\frac{s_1}{c},\;\;\;
\left.\frac{2e^{-2\beta_r(\xi)\wid{L}}}{1-e^{-2\beta_r(\xi)\wid{L}}}\right|_{\xi=0}
\le\frac{2e^{-2c^{-1}\ov{L}}}{1-e^{-2c^{-1}\ov{L}}} .
\enn
In addition, $\beta_r(\xi)\ra +\infty$ as $\xi\ra\infty$, and so ${2e^{-2\beta_r(\xi)\wid{L}}}/[{1-e^{-2\beta_r(\xi)\wid{L}}}]\ra 0$ as $\xi\ra\infty$.
It is then concluded that
\ben
\sup\limits_{\xi\in\R^2}\frac{2e^{-2\beta_r(\xi)\wid{L}}}{1-e^{-2\beta_r(\xi)\wid{L}}}
\le\frac{2e^{-2c^{-1}\ov{L}}}{1-e^{-2c^{-1}\ov{L}}}.
\enn
This, together with (\ref{4.24}), implies the required estimate (\ref{thm4.3a}).
\end{proof}

For $\bm\om:=(\ch{p},\ch{\bm u})\in H$ and $\bm\theta:=(q,\bm{v})\in H$ it easily follows
by the definition of the sesquilinear forms $a(\cdot,\cdot)$ and $a_p(\cdot,\cdot)$ that
\be\label{dtnerror}
&&|a(\bm\om,\bm\theta)-a_p(\bm\om,\bm\theta)|\no\\
&&=\left|\int_{\G_h}s^{-1}\ov{ q}(\mathscr{B}-\mathscr{B}_{pml})\bm\ch{p}d\g\right|\no\\
&&\le |s|^{-1}\|q\|_{H^{1/2}(\G_h)}\|\mathscr{B}-\mathscr{B}_{pml}\|_{L(H^{1/2}(\G_h),H^{-1/2}(\G_h))}
\|\bm\ch{p}\|_{H^{1/2}(\G_h)}\no\\
&&\le |s|^{-1}\left[1+(h-f_+)^{-1}\right]\|q\|_{H^1(\Om_h)}
\|\mathscr{B}-\mathscr{B}_{pml}\|_{L(H^{1/2}(\G_h),H^{-1/2}(\G_h))}\|\bm\ch{p}\|_{H^{1}(\Om_h)}\no\\
&&\le |s|^{-1}\left[1+(h-f_+)^{-1}\right]
\|\mathscr{B}-\mathscr{B}_{pml}\|_{L(H^{1/2}(\G_h),H^{-1/2}(\G_h))}\|\bm\theta\|_{H}\|\bm\om\|_{H},
\en
where we have used the trace theorem (see \cite[Lemma 2.2]{GLZ2017}) to get the second inequality.
Using (\ref{dtnerror}) and Theorem \ref{thm4.3}, we can now prove the exponential convergence of
the PML method.

\begin{theorem}\label{convergence}
Let $(p,\bm u)$ be the solution of the problem $(\ref{2.8})$ and let $(p_{pml},\bm u_{pml})$ be
the solution of the truncated PML problem $(\ref{tdpml})$ with $s_1=1/T$ in the time-domain.
Then, under the assumptions $(\ref{assumption})$ and $(\ref{assumption1})$,
we have the error estimate
\be\label{result}
&&\int_0^T(\|p-p_{pml}\|^2_{H^1(\Om_h)}+\|\bm u-\bm u_{pml}\|^2_{H^1(\Om)^3})dt\no\\
&&\lesssim\max\{1,T^2\}(T^4+T^2)(\g_1+\g_2)(1+\sig_0T)^2\frac{e^{-4\sig_0L/c}}{(1-e^{-2\sig_0L/c})^2}
\|g\|^2_{H^3(0,T;L^2(\Om_h))}\;\;
\en
where $\g_1$ and $\g_2$ are positive constants which are independent of $(p,\bm u)$ and
$(p_{pml},\bm u_{pml})$ but may depend on $T$.
\end{theorem}

\begin{proof}
First, let $\bm{\om}=(\ch{p},\ch{\bm u})$ and $\bm{\om}_p=(\ch{p}_{pml},\ch{\bm u}_{pml})$
be the solutions of the variational problems (\ref{3.2}) and (\ref{4.18}), respectively.
Then, by (\ref{dtnerror}) we have
\ben\label{4.25}
&&|a(\bm\om-\bm\om_p,\bm\om-\bm\om_p)|\no\\
&&=|a(\bm\om,\bm\om-\bm\om_p)-a(\bm\om_p,\bm\om-\bm\om_p)|\no\\
&&=|a_p(\bm\om_p,\bm\om-\bm\om_p)-a(\bm\om_p,\bm\om-\bm\om_p)|\no\\
&&\le |s|^{-1}\left[1+(h-f_+)^{-1}\right]
\|\mathscr{B}_{pml}-\mathscr{B}\|_{L(H^{1/2}(\G_h),H^{-1/2}(\G_h))}
\|\bm\om_p\|_{H}\|\bm\om-\bm\om_p\|_{H}.\qquad
\enn
This, together with Theorem \ref{thm4.3} and the uniform coercivity of $a(\cdot,\cdot)$
(see (\ref{coercivity})), implies that
\ben
\|\bm\om-\bm\om_p\|_{H}\le C^{-1}|s|^{-1}\left[1+(h-f_+)^{-1}\right]C_U(s,\ov{L})\|\bm\om_p\|_{H},
\enn
where $C$ is defined in (\ref{coercivity}).
By the Parseval identity (\ref{A.5}) and the definition of $C_U(s,\ov{L})$ it is deduced that
\ben\label{parseval}
&&\int_{0}^{\infty}e^{-2s_1t}\|\mathscr{L}^{-1}(\bm\om-\bm\om_p)\|_{H}^2dt\no\\
&&\quad=\frac{1}{2\pi}\int_{-\infty}^{\infty}\|\bm\om-\bm\om_p\|_{H}^2ds_2\no\\
&&\quad\le\frac{1}{\pi}\int_{0}^{\infty}
\frac{\left[1+(h-f_+)^{-1}\right]^2\max\{1,(|s|/c)^2\}}{C^2|s|^2}
\frac{4e^{-4\ov{L}/c}}{(1-e^{-2\ov{L}/c})^2}\|\bm\om_p\|_{H}^2ds_2.\quad
\enn
This gives
\be\label{thm3.4a}
&&\int_0^T(\|p-p_{pml}\|^2_{H^1(\Om_h)}+\|\bm u-\bm u_{pml}\|^2_{H^1(\Om)^3})dt\no\\
&&\quad\le\int_0^Te^{-2s_1(t-T)}\left[\|p-p_{pml}\|^2_{H^1(\Om_h)}
+\|\bm u-\bm u_{pml}\|^2_{H^1(\Om)^3}\right]dt\no\\
&&\quad\le e^{2s_1T}\int_{0}^{\infty}e^{-2s_1t}\left[\|p-p_{pml}\|_{H^1(\Om_h)}^2
+\|\bm u-\bm u_{pml}\|_{H^1(\Om)^3}^2\right]dt\no\\
&&\quad=e^{2s_1T}\int_{0}^{\infty}e^{-2s_1t}\|\mathscr{L}^{-1}(\bm\om-\bm\om_p)\|_{H}^2dt\no\\
&&\quad\le\frac{e^{2s_1T}}{\pi}\frac{4e^{-4\ov{L}/c}}{(1-e^{-2\ov{L}/c})^2}
\int_{0}^{\infty}\frac{\left[1+(h-f_+)^{-1}\right]^2\max\{1,(|s|/c)^2\}}{C^{2}|s|^{2}}
\|\bm\om_p\|_{H}^2ds_2.\qquad
\en
Since $s_1>0$ is arbitrarily fixed, and by the definition of $C$ (see (\ref{coercivity})),
there exists a sufficiently large positive constant $M$ such that
\be\label{gamma1}
\frac{\left[1+(h-f_+)^{-1}\right]^2\max\{1,(|s|/c)^2\}}{C^{2}|s|^{2}}\le\g_1|s|^4
\en
for $s_2\ge M$, where $\g_1$ is a constant independent of $s_2$. On the other hand, it is
easy to see that
\be\label{gamma2}
\frac{\left[1+(h-f_+)^{-1}\right]^2\max\{1,(|s|/c)^2\}}{C^{2}|s|^{2}}\le\g_2
\en
for $0\le s_2\le M$, where $\g_2$ is a constant independent of $s_2$.
By (\ref{gamma1}), (\ref{gamma2}), assumption (\ref{assumption}), Theorem \ref{thm4.1} and the Parseval identity (\ref{A.5})
we obtain that
\ben
&&\int_{0}^{\infty}\frac{\left[1+(h-f_+)^{-1}\right]^2\max\{1,(|s|/c)^2\}}{C^{2}|s|^{2}}
\|\bm\om_p\|_{H}^2ds_2\\
&&\quad\le\int_{0}^{M}\g_2\|\bm\om_p\|_{H}^2ds_2
+\int_{M}^{\infty}\g_1|s|^4\|\bm\om_p\|_{H}^2ds_2\\
&&\quad\le\int_{0}^{M}\g_2(1+s_1^{-1}\sig_0)^2\left[\frac{1+2s_1^2}{s_1^4\min\{1,s_1^2\}}\|\ch{g}\|_{L^2(\Om_h)}^2
+s_1^{-2}\|s\ch{g}\|_{L^2(\Om_h)}^2\right]ds_2\\
&&\quad\;\;+\int_{M}^{\infty}\g_1(1+s_1^{-1}\sig_0)^2\left[\frac{1+2s_1^2}{s_1^4\min\{1,s_1^2\}}
\|s^2\ch{g}\|_{L^2(\Om_h)}^2+s_1^{-2}\|s^3\ch{g}\|_{L^2(\Om_h)}^2\right]ds_2\\
&&\quad\le C_0\int_0^{\infty}\left[\|\ch{g}\|_{L^2(\Om_h)}^2+\|s\ch{g}\|_{L^2(\Om_h)}^2
+\|s^2\ch{g}\|_{L^2(\Om_h)}^2+\|s^3\ch{g}\|_{L^2(\Om_h)}^2\right]ds_2\\
&&\quad=\pi C_0\int_0^{\infty}e^{-2s_1t}\left[\|g\|_{L^2(\Om_h)}^2+\|\pa_tg\|_{L^2(\Om_h)}^2
+\|\pa_t^2g\|_{L^2(\Om_h)}^2+\|\pa_t^3g\|_{L^2(\Om_h)}^2\right]dt,
\enn
where
\ben
C_0=(1+s_1^{-1}\sig_0)^2\frac{1+2s_1^2}{s_1^4\min\{1,s_1^2\}}(\g_1+\g_2).
\enn
By this inequality and (\ref{thm3.4a}) the required estimate (\ref{result}) follows easily
on taking $s_1=T^{-1}$ and using the assumption (\ref{assumption1}), where integer $m\geq 1$
should be chosen small enough to ensure the rapid convergence (thus we need to take $m = 1$) 
noting the definition of $\ov{L}=\sig_0L/(m+1)$.
The proof is thus complete.
\end{proof}

\begin{remark}\label{re5} {\rm
Theorem \ref{convergence} implies that, for large $T$ the exponential convergence of the PML method 
can be achieved by enlarging the thickness $L$ or the PML absorbing parameter $\sigma_0$ 
which increases as $\ln T$.
}
\end{remark}

\section{Conclusion}\label{sec4}

This paper studied the time-dependent acoustic-elastic interaction problem associated with
a bounded elastic body immersed in a homogeneous air or fluid above a rough surface.
A time-domain perfectly matched layer (PML) is introduced to truncate the unbounded domain of
the interaction problem above a finite layer in the $x_3$ direction containing the elastic body,
leading to a PML problem in a finite strip domain.
The PML layer is constructed by the real coordinate stretching technique associated with
$[\Rt(s)]^{-1}$ in the Laplace domain, where $s\in\C_+$ is the Laplace transform variable.
The well-posedness and stability estimate of the PML problem are established,
based on the Laplace transform and a variational method.
Moreover, the exponential convergence of the PML method has also been proved in terms
of the thickness and parameters of the PML layer.

In practical computation, the PML problem obtained in this paper and defined in a strip domain
must be truncated as well in the $x_1$ and $x_2$ directions, which may be achieved by constructing a rectangular or cylindrical PML layer in the strip domain. Further, our method can be extended to
other time-dependent scattering problems, such as diffraction gratings, time-domain elastic scattering
by rough surfaces in $\R^2$ and even more complicated electromagnetic-elastic interaction problems
in unbounded layered structures. We hope to report such results in the future.

\section*{Acknowledgements}

This work was partly supported by the NNSF of China grants 91630309 and 11771349.
The first author thanks Dr. Tielei Zhu for constructive discussions and suggestions on this paper.

\appendix
\renewcommand\theequation{A.\arabic{equation}}

\section{Laplace transform}\label{ap1}

For each $s\in\C_+$, the Laplace transform of the vector function $\bm{u}(t)$ is defined as
\ben
\ch{\bm{u}}(s)=\mathscr{L}(\bm{u})(s)=\int_0^{\ify}e^{-st}\bm{u}(t)dt.
\enn
The Fourier transform of $\phi(\wid{x},x_3)$ is defined by
\ben
\hat{\phi}(\xi,x_3)=\mathscr{F}(\phi)(\xi,x_3)
=\int_{\R^2}e^{-i\wid{x}\cdot\xi}\phi(\wid{x},x_3)d\wid{x},\;\;;\xi\in\R^2
\enn
and the inverse Fourier transform of $\wi{\phi}(\xi)$ is
\ben
\phi(\wid{x},x_3)=\mathscr{F}^{-1}(\wi{\phi})(\wid{x},x_3)
=\frac{1}{(2\pi)^2}\int_{\R^2}e^{i\wid{x}\cdot\xi}\wi{\phi}(\xi,x_3)d\xi.
\enn
The Laplace transform has the following properties:
\be\label{A.1}
\mathscr{L}(\bm{u}_t)(s)=s\mathscr{L}(\bm{u})(s)-\bm{u}(0),\\ \label{A.2}
\mathscr{L}(\bm{u}_{tt})(s)=s^2\mathscr{L}(\bm{u})(s)-s\bm{u}(0)-\frac{d\bm{u}}{dt}(0),\\ \label{A.3}
\int_0^t\bm{u}(\tau)d\tau=\mathscr{L}^{-1}(s^{-1}\ch{\bm{u}}(s))(t),
\en
where $\mathscr{L}^{-1}$ denotes the inverse Laplace transform.

By the definition of the Fourier transform we have that for any $s_1>0$,
\ben
\mathscr{F}(\bm{u}(\cdot)e^{-s_1\cdot})(s_2)&=&\int_{-\infty}^{+\infty}\bm{u}(t)e^{-s_1t}e^{-is_2t}dt
=\int_{0}^{\infty}\bm{u}(t)e^{-(s_1+is_2)t}dt\\
&=&\mathscr{L}(\bm{u})(s_1+is_2),\;\;\;s_2\in\R.
\enn
From the formula of the inverse Fourier transform it is easy to verify that
\ben
\bm{u}(t)e^{-s_1t}=\mathscr{F}^{-1}\{\mathscr{F}(\bm{u}(\cdot)e^{-s_1\cdot})\}
=\mathscr{F}^{-1}\left(\mathscr{L}(\bm{u}(s_1+is_2))\right),
\enn
which implies that
\be\label{A.4}
\bm{u}(t)=\mathscr{F}^{-1}\left(e^{s_1t}\mathscr{L}(\bm{u}(s_1+is_2))\right),
\en
where $\mathscr{F}^{-1}$ denotes the inverse Fourier transform with respect to $s_2$.

By (\ref{A.4}), the Plancherel or Parseval identity for the Laplace transform can be obtained
(see \cite[(2.46)]{Cohen2007}).

\begin{lemma}{\rm (Parseval identity)}
If $\ch{\bm{u}}=\mathscr{L}(\bm{u})$ and $\bm{v}=\mathscr{L}(\bm{v})$, then
\be\label{A.5}
\frac{1}{2\pi}\int_{-\ify}^{\ify}\ch{\bm{u}}(s)\cdot\bm{v}(s)ds_2
=\int_0^{\ify}e^{-2s_1t}\bm{u}(t)\cdot\bm{v}(t)dt.
\en
for all $s_1>\la$, where $\la$ is the abscissa of convergence for the Laplace transform
of $\bm{u}$ and $\bm{v}$.
\end{lemma}

\begin{lemma} {\rm\cite[Theorem 43.1]{treves1975}.}
Let $\ch{\bm{\om}}(s)$ denote the holomorphic function in the half plane $s_1>\sig_0$,
valued in the Banach space $\E$. The following statements are equivalent:
\begin{enumerate}[1)]
\item there is a distribution $\om\in\mathcal{D}_+^\prime(\E)$ whose Laplace transform is
equal to $\ch{\bm{\om}}(s)$, where $\mathcal{D}_+^\prime(\E)$ is the space of distributions
on the real line which vanish identically in the open negative half line;
\item there is a $\sig_1$ with $\sig_0\leq\sig_1<\infty$ and an integer $m\geq0$ such that
for all complex numbers $s$ with $s_1=\Rt(s)>\sig_1$, it holds that
$\|\ch{\bm{\om}}(s)\|_{\E}\lesssim(1+|s|)^m$.
\end{enumerate}
\end{lemma}

\end{document}